\documentclass[12pt,oupthm]{article}
\usepackage{amsmath,amssymb,amsfonts,amsthm}
\usepackage{newtxtext,newtxmath}
\usepackage{tgbonum}
\usepackage[utf8]{inputenc}
\usepackage[english]{babel}
\usepackage[aticle]{ragged2e}
\usepackage[square,numbers]{natbib}
\usepackage{microtype}
\usepackage[svgnames]{xcolor} 
\usepackage[hang, small, labelfont=bf, up, textfont=it]{caption} 
\usepackage{booktabs} 
\usepackage{lastpage} 
\usepackage{graphicx} 
\usepackage{sectsty} 
\allsectionsfont{\usefont{OT1}{phv}{b}{n}} 
\usepackage{geometry}
\usepackage[T1]{fontenc} 
\usepackage[utf8]{inputenc} 
\usepackage{hyperref}
\usepackage{blindtext} 
\usepackage[sc]{mathpazo} 
\usepackage[T1]{fontenc} 
\usepackage{microtype} 
\usepackage[english]{babel} 
\usepackage{abstract}

\usepackage{titlesec} 
\renewcommand\thesection{\Roman{section}} 
\renewcommand\thesubsection{\roman{subsection}} 
\titleformat{\section}[block]{\large\scshape\centering}{\thesection.}{1em}{} 
\titleformat{\subsection}[block]{\large}{\thesubsection.}{1em}{} 
\usepackage{titling} 
\usepackage{hyperref} 
\newcommand{\myul}[2][black]{\setulcolor{#1}\ul{#2}\setulcolor{black}}
\usepackage{float}
\usepackage{multicol}
\usepackage{graphicx}
\usepackage{mathtools}
\usepackage{xcolor}
\usepackage{seqsplit}
\usepackage{babel,booktabs}
\usepackage{array}
\usepackage{xcolor, soul}
\sethlcolor{green}
\usepackage{color}
\usepackage[most]{tcolorbox}
\usepackage{blindtext}
\usepackage{graphicx}
\DeclareGraphicsExtensions{.pdf,.png,.jpg}
\usepackage[font=small,labelfont=bf]{caption}
\usepackage{tikz}
\usepackage{graphicx}
\usepackage{epstopdf}
\usepackage{mathtools}
\usepackage{colortbl}
\usepackage{caption}
\usepackage{amsthm}
\usepackage{listings}
\usepackage{color}
\usepackage{array}
\usepackage{enumerate}
\usepackage{enumitem}
\usepackage{colortbl}
\usepackage{subcaption}
\usepackage{float}
\theoremstyle{definition}

\usepackage{color, colortbl}
\usepackage{multirow}
\usepackage{multicol}
\usepackage[export]{adjustbox}
\usepackage{rotating}
\usepackage{graphics,graphicx,lineno,amsthm,subcaption,float,url}
\usepackage{physics}
\usepackage{hyperref}\hypersetup{colorlinks,citecolor=blue}
\hypersetup{breaklinks=true}
\newtheorem{thm}{Theorem}
\newtheorem{lem}{Lemma}
\newtheorem{Cor}{Corollary}
\newtheorem{Prop}{Proposition}

\usepackage{listings}
\begin{document}
\justifying
\begin{center}
{\Large  \textbf {Binomial Coefficients in a Row of Pascal’s Triangle from Extension of Power of Eleven: Newton's Unfinished Work}}\\[4.5ex]
%%%%%%%%%%%%%%%%%%%%%%%%%%%%%%%%%%%%%%%%%%%%%%
{\bf Md. Shariful Islam}\textsuperscript{1}  {\bf Md. Robiul Islam}\textsuperscript{3} {\bf Md.  Shorif Hossan}\textsuperscript{2}  and {\bf Md.  Hasan Kibria}\textsuperscript{1} \\[0.65ex]
\normalsize{\textsuperscript{1}Department of Mathematics, University of Dhaka, Bangladesh}\\[0.2ex]
\normalsize{\textsuperscript{2}Department of Applied Mathematics, University of Dhaka, Bangladesh}\\[0.2ex]
\normalsize{\textsuperscript{3}Department of Computer Science and Engineering, Green University of Bangladesh, Dhaka, Bangladesh}\\[0.35ex]
\normalsize{ {\bf E-mail:} \texttt{mdsharifulislam@du.ac.bd}, \texttt{robiul258@gmail.com}, \texttt{shorif@du.ac.bd}, and \texttt{hasan-2017713694@math.du.ac.bd}}\\[0.55ex]
%\normalsize{ {\bf ORCID:}
%\texttt{0000-0003-4390-7262}, %\texttt{0000-0003-4115-9002}}\\[4ex]
\hspace{-2.0in}\textit{\textsuperscript{$\ast$}Author for correspondence}: robiul258@gmail.com

\end{center}
\begin{abstract}
\noindent The aim of this paper is to find a general formula to generate any row of Pascal's triangle as an extension of the concept of $\left(11\right)^{n}$. In this study, the visualization of each row of Pascal’s triangle has been presented by extending the concept of the power of 11 to the power of 101, 1001, 10001, and so on. We briefly discuss how our proposed concept works for any $n$ by inserting an appropriate number of zeros between $1$ and $1$ (eleven), that is the concept of $\left(11\right)^{n}$ has been extended to $\left(1\Theta1\right)^{n}$, where $\Theta$ represents the number of zeros. We have proposed a formula for obtaining the value of $\Theta$. The proposed concept has been verified with Pascal's triangle and matched successfully. Finally,   Pascal's triangle for a large n has been presented considering the $51^{\text{st}}$ row as an example.
\end{abstract}


\section{Introduction} \label{sec:intro}
Algebra is a spacious part of the science of mathematics which provides the opportunity to express mathematical ideas precisely. In algebra, the Binomial expansion and Pascal's triangle are considered important.
Pascal's triangle is an arrangement of the binomial coefficients and one of the most known integer models.
Though it was named after French scientist Blaise Pascal, it was studied in ancient India, Persia, China, Germany, and Italy \cite{edwards}.\\[2ex]
In reality, the definition of the triangle was made centuries ago. In 450 BC, an Indian mathematician named Pingala is said to have introduced the definition of this triangle in a Sanskrit poetry book. At the same time, the commentators of this book acquainted with the diagonal surface of the triangle, which is the sum of the Fibonacci numbers. Chinese mathematicians had the same idea and named the triangle as ``Yang Hui's triangle”. Later, Persian mathematician Al-Karaji and Persian astronomer-poet Omar Khayyam named the triangle as the “Khayyam triangle”.
It also has multi-dimensional shapes, the three-dimensional shape is referred to as Pascal's pyramid or Pascal's tetrahedron, while the other general-shaped ones are called Pascal's simplifications.\\[2ex] 
Various studies have been conducted in many different disciplines about Pascal's triangle. For the construction of Pascal's triangle, Sgroi \cite{Sgroi} stated that each line starts with 1 and ends with 1, and this series can be expanded with simple cross-joints. Jansson \cite{jansson} developed three geometric forms related to Pascal's triangle and included examples on each form. Toschi \cite{toschi} used various permutations to generate new forms of Pascal's triangles and generalized them. Duncan and Litwiller \cite{duncan} addressed the reconstruction of Pascal's triangle with the individuals. Here they collected data on the opinions of individuals using qualitative methods, and determined the methods of constructing the Pascal's triangle in different ways with the attained findings.\\[2ex]
Researches worked on Pascal's fascinating characteristics.  Using the principle of permutation, Putz \cite{putz} designed the Pascal Polytope and linked it to the Fibonacci concept. Houghton \cite{houghton} gave the concept of the relationship between  successive differential operation of a function and Pascal's triangle. With an application, he attempted to incorporate the idea of a differentiable function into Pascal's triangle. The relationship between Pascal's triangle and the Tower of Hanoi has been elucidated by Andreas M Hinz \cite{hinz}. Finding diagonal sum \cite{hoggatt}, k-Fibonacci sequence \cite{falcon}, recurrence relations \cite{green}, finding exponential $(e)$ \cite{brothers} were a part of those to describe the work that generates from the Pascal's triangle. Some fascinating properties of Pascal’s triangle are available in \cite{Bondarenko, Korec}. In 1956, Freund \cite{Freund} elicited that the generalized Pascal's triangles of $s^\text{th}$  order can be constructed from the generalized binomial coefficients of order $s$. Bankier \cite{bankier} gave the  Freud’s alternative proof. Kall{\'o}s  generalized Pascal's triangle from algebraic point of view by different bases. He tried to generalize Pascal's triangle using the power of integers \cite{Kall_1}, powers of base numbers \cite{Kall_2} and their connections with prime number \cite{Farkas}. kuhlmann tried to generate Pascal's triangle using the T-triangle concept \cite{kuhlmann}.\\[2ex]
The concept of the power of $11$ was first introduced by Sir Isaac Newton. He noticed that first five  rows of Pascal's triangle are formed by the power of $11$ and claimed (without proof) that subsequent rows can also be generated by the power of eleven as well \cite{newton1736}.  Arnold \textit{et al.} \cite{Arnold} showed if one assigns  a place value to each of the individual terms in a certain row of the triangle, the pattern can be seen again. Morton \cite{Morton} noted the Pascal's triangle property by the power of 11 for 10 base numerals system. Mueller \cite{Mueller} noted that one can get the $n^\text{th}$ power of $11$ from the $n^\text{th}$ row of the Pascal's triangle with positional addition.\\[2ex]
It is clearly concluded that above mentioned works did not express the full row of Pascal's triangle from the power of $11$, or some variant of it, as Sir Isaac Newton hinted . This paper has worked on the generalization of Pascal's triangle by extending the power of eleven idea. Here, we have extend the concept of power of $11$ to the power of $101, 1001, 10001, \ldots$ and proved a general formula to achieve any row of Pascal's triangle. Using our formula one can generate any row of Pascal's triangle, regardless of the number of rows one can imagine.

\section{Methods} \label{sec: Methods}
The very basic definition to get any element of a row of the Pascal's triangle is the summation of two adjacent elements of the previous row. Each number in Pascal's triangle is the sum of two numbers above that number. Usually, the lines of Pascal's triangle are numbered starting from $n = 0$ on the top and the numbers in each line are starting from $k = 0$ on the left. For $k=0$, their is only one value 1. As the next lines are created, The remaining right most and left most element for new row is 1.

\[
\begin{array}{ccccccccccccc} {}  & {} & {} &{}  & {} & {} & {1} & {}  & {} & {} &{}  & {} & {}  \\ {}  & {} & {} &{}  & {} & {1} & {} & {1}  & {} & {} &{}  & {} & {} \\ {}  & {} & {} &{}  & {1} & {} & {2} & {}  & {1} & {} &{}  & {} & {} \\ {}  & {} & {} &{1}  & {} & {3} & {} & {3}  & {} & {1} &{}  & {} & {}\\ {}  & {} & {1} &{}  & {4} & {} & {6} & {}  & {4} & {} &{1}  & {} & {}\\ {}  & {1} & {} &{5}  & {} & {10} & {} & {10}  & {} & {5} &{}  & {1} & {}\\ {1}  & {} & {6} &{}  & {15} & {} & {20} & {}  & {15} & {} &{6}  & {} & {1} \\  {}  & {} & {\ldots} &{}  & {} & {} & {\ldots} & {}  & {} & {} &{\ldots}  & {} & {} \end{array}
\]
\newline
\vspace{0.5cm}
 \hspace{5cm}\textbf{Figure 1:} Pascal's triangle\\[1ex]

The concept of the power of $11$ leads to us $11^{1}=11$, $1^{\text{st}}$ row of Pascal's triangle and so $11^{2}=121$, $11^{3}=1331$ and $11^{4}=14641$ reveal $2^{\text{nd}}$, $3^{\text{rd}}$ and $4^{\text{th}}$ row respectively. Before finding the general rule for subsequent rows, we first elaborate the previous concept of power 11. The reason behind getting Pascal's triangle by the power of 11 lies on the general rule of multiplication.   What do we get from multiplication of a number by $11?$\\

\[
\begin{array}{@{}r@{}}
2^{nd}\text{ row of Pascal's triangle} \rightarrow 121\\
{}\times 11\\
\hline
1${\color{red}2}$1 \\
\text{left shift of all digits by 1 place} \rightarrow 12${\color{red}1}$0 \\
\hline
3^{rd}\text{ row of Pascal's triangle} \rightarrow  13$\myul[red]{3}$1\\
\end{array}
\]
\newline
\vspace{0.5cm}
\hspace{4cm} \textbf{Figure 2:} Results after multiplication by $11$\\[1ex]

Figure 2 shows that multiplication of a number by $11$ gives an output which is the sum of the two adjacent numbers of previous row of Pascal's triangle. 
\\\\
Patently $11^{5}=161051$ and  $11^6=1771561$, but the $5^{\text{th}}$ and $6^{\text{th}}$ row of Pascal's triangle are 
\begin{center}
 {1} { } { } {5}  { } { }  {10} { } { } {10}  { } { }  {5} { } { } {1}  {} \\
 {}   { } { }  {}   { } { }  {}  { } { }  {} { } { }  {}  { $\qquad~~\text{and}$} { }   {}  { } { }   {}  {}  \\
   {1}   { } { }  {6}   { } { }  {15}  { } { }  {20}    { } { }  {15}  { } { }   {6}  { } { }   {1}  {}\\
   \end{center}
respectively. The above scheme fails for $11^5$ or $11^6$. Why are we not getting the $5^{th}$ row or why does the power of $11$ fail here? The answer is the middle values of the $5^{\text{th}}$ row of Pascal's triangle are two digit numbers, whereas the power of $11$ represents Pascal's row as a representation of one decimal place. So for finding $5^{\text{th}}$ or any row onward, we need a formula that can represent the number generated from the power of 11 have two or more digits. Now, we will endeavor to formulate a specific rule that generates the required number of digits for the representation of a row of the Pascal's triangle.\vspace{0.2cm}
\newline
At first, we attempt to  generate the number of two digits using the very basic rules of multiplication. Figure 3, displays the impact of multiplication by $101.$\\[2ex]

 \[
\begin{array}{@{}r@{}}
101 \\
{}\times 101\\
\hline
${\color{red}1}$01\\
\text{zeros cause the left shift of all digits by 1 place} \rightarrow $~{\color{red}00}$00\\
\text{left shift of all digits by 2 places} \rightarrow 1${\color{red}01}$00\\
\hline
1$\myul[red]{02}$01\\
{}\times 101\\
\hline
${\color{blue}1}{\color{green}02}$01\\
${\color{blue}00}{\color{green}00}$00\\
\text{left shift of all digits by 2 places} \rightarrow 1${\color{blue}02}{\color{green}01}$00\\
\hline 
1$\myul[blue]{03}\myul[green]{03}$01\\ 
\end{array}
\]
\vspace{0.5cm}
\hspace{4cm}  \textbf{Figure 3:} Results after multiplication by $101$\\[1ex]
The underlined numbers are same as the summation of two adjacent numbers of the previous row, but multiplication by 101 displays the rows as a representation of two digits number.\\[1.05ex]
Now, $101^{5}=10510100501$, from which we can construct $5^{\text{th}}$ row of Pascal's triangle by omitting extra zeros and separating the digits.
\begin{center}
 {1}  {} {5}   {}  {10}  {} {10}   {}  {5} {}   {1}  {}   
\end{center}
Similarly from $101^{6}=1061520150601$ and $101^{7}=107213535210701$, we can easily construct $6^{\text{th}}$ and $7^{\text{th}}$ row respectively.\\[0.1ex]
\begin{center}
{1}   { } { }  {06}   { } { }  {15}  { } { }  {20}    { } { }  {15}  { } { }   {06}  { } { }   {01}  {}\\
 { $\quad\text{and}$}\\
{1}   { } { }  {07}   { } { }  {21}  { } { }  {35}    { } { }{35}  { } { }  {21}  { } { }   {07}  { } { }   {01}  {}
\end{center}
$101^{5}$, $101^{6}$ and $101^{7}$ all are representing $5^{th}$, $6^{th}$ and $7^{th}$ row of Pascal's triangle respectively as a representation of two digit numbers due to the insertion of one zero between $1$ and $1~$ in 11 such that $101$. $11^5=161051$, $11^6 = 1771561$ and $11^7=19487171$ could also represent the respective rows according to the Newton's claim but $101^n$ makes it precise.\\ 
Can a conclusion be drawn for generating any row of Pascal's triangle with the help of extended concept of the power of 11 such as $101^n$? 
The $9^{\text{th}}$ row of Pascal's triangle is 
$$1~~ 9~~ 36 ~84~~ 126~~ 126 ~~84 ~~36~ 9~~ 1$$
Plainly, $101^9 =1093685272684360901$ does not give the $9^{\text{th}}$ row because of the central element of this row contains three digits.

So the representation of three decimal places for each entry of Pascal's triangle requires a new formula to be  generated.
The previous context directed that multiplication of a number by $11$ and $101$ makes the left shift of all digits by one and two places respectively. Therefore three digits representation requires multiplication by  $1001$.\\[1.25ex]
Figure 4, proofs the left shift of all digits by  $3$ times when a number is multiplied by  $1001.$\\

\[
\begin{array}{@{}r@{}}
1001 \\
{}\times 1001\\
\hline
${\color{red}1}$001\\
${\color{red}00}$000\\
${\color{red}000}$000\\
\text{left shift of all digits by 3 time} \rightarrow 1${\color{red}001}$000\\
\hline
1$\myul[red]{002}$001\\ 
{}\times 1001\\
\hline
${\color{blue}1}{\color{green}002}$001\\
${\color{blue}00}{\color{green}000}$000\\
${\color{blue}000}{\color{green}000}$000\\
\text{left shift of all digits by 3 time} \rightarrow 1${\color{blue}002}{\color{green}001}$000\\
\hline
1$\myul[blue]{003}\myul[green]{003}$001 
\end{array}
\]
\vspace{0.5cm}
\hspace{3.5cm}  \textbf{Figure 4:} Results after multiplication by $1001$\\[1ex]

By continuing the multiplication by 1001 in Figure 4, we get $$1001^9 =1009036084126126084036009001$$ 
from which one may form the $9^{th}$ row
of Pascal's triangle by partitioning the digits of the number from the right three digits in each partition. 
\begin{center}
   $1~~009~~036~~084~~126~~126~~084~~036~~009~~001$
\end{center}
Similarly,  $(1001)^{10}=1010045120210252210120045010001$, 
$$1010045120210252210120045010001\longmapsto 1~~010~~045~~120~~210~~252~~210~~120~~045~~010~~001$$
the $10^{\text{th}}$ row of the Pascal's triangle.
\section{Results and discussion}
From the above study, it may be concluded that the representation of three decimal places requires the left shift of all digits by three places, and the left shift of all digits requires two zeros between $1$ and $1~$ in 11, that is $1001$. Why do we require three decimal places representation for $9^{th}$ and $10^{th}$ rows of Pascal's triangle? Because the central elements of  $9^{\text{th}}$ and $10^{\text{th}}$ rows are of three digits. Similarly, we need two digits representation for $5^{th}$ to $8^{th}$ rows since the central element of these rows are numbers of two digits. And, the first four rows satisfy $11^n$ since the central element of the first four rows  contains one digit only. So for any row, the number of decimal places representation should be equal to the number of digits in the central value of that row.\\[1.5ex]
The above discussion compels to generate a formula to find the central value of any row of the Pascal's triangle. 
For an odd number, say $n=9,$ we get $n+1=10$ elements in $9^{th}$ row and so the central value should be $\left(\frac{10}{2}\right)^{th}=5^{th}$ observation of that row, which is $\binom{9}{5-1}=\binom{9}{4}=126$. 
For an even number, say $n=10$ we get $n+1=11$ elements and the central value should be $\lceil\frac{11}{2}\rceil^{th}=6^{th}$  observation, which is $\binom{10}{6-1}=\binom{10}{5}=252$.\\[1.25ex]
By taking the \textit{floor} value of $\frac{n}{2}$, a formula for central value of $n^{th}$ row is $\binom{n}{\lfloor\frac{n}{2}\rfloor}$. But we never need a central element rather it is necessary to know how many digits the central element has. Applying the property of Logarithmic function, one can identify how many digits (or decimal places) of the central element has without knowing it.
Therefore the number of digits in the central value is given by
$$\lceil\log_{10} \binom{n}{\lfloor\frac{n}{2}\rfloor}\rceil$$ %Since $ceil\left(\log_{10}(X)\right)$ represents the number of digits of $X$. 
Since $\lceil\log_{10} (X)\rceil$ represents the number of digits of $X$ when $X \ne 10^{n}$, for $n \in \mathbb{N}$. For a central value of $d$ decimal places we require  $d-1$ zeros between $1~\text{and}~1~$ in 11, such that $\left(1~\left(d-1\right)~zeros~1\right)^{n}$. So, the required number of zeros between $1~\text{and}~1~$ in 11 can be obtained by taking the \textit{floor} value of $\log_{10} \binom{n}{\lfloor\frac{n}{2}\rfloor}$.\\
If $\Theta$ represents the number of zeros between $1$ and $1~$ in 11. Then $$\Theta=\lfloor\log_{10} \binom{n}{\lfloor\frac{n}{2}\rfloor}\rfloor$$. We now verify it for an odd number $n=9$ and an even number $n=10$.\\ [1ex]
If, $n=9$ then $\Theta=2$, and if $n=10$ then $\Theta=2$.\\[1ex] 
For both of the numbers we need $2$ zeros between  $1$ and $1~$ in 11. So, to get $9^{th}$ and $10^{th}$ rows we have to calculate $1001^{9}$ and $1001^{10}$ respectively. Both of these cases have been discussed above.\\[1ex]
It's time to generate the formula to find any row of Pascal's triangle. The general formula for generating $n^{th}$ row of Pascal's triangle is $(1\Theta1)^n$.
For a random number such as $n=15$ we get $\Theta=3$.\\[1ex] 
So, we have to insert 3 zeros and the $15^{th}$ row can be constructed from the following 
$$10001^{15}= 
1001501050455136530035005643564355005300313650455010500150001$$
Partitioning the digits from the right four digits at a time as shown below
$$1~0015~0105~0455~1365~3003~5005~6435~6435~5005~3003~1365~0455~0105~0015~0001$$
Notice the partitioning yields the $15^{\text{th}}$ row of the Pascal's triangle
$$1~15~105~455~1365~3003~5005~6435~6435~5005~3003~1365~455~105~15~1$$
Similarly, we may verify for $n=16$, $\Theta=4$ and\\[2ex]
$(100001)^{16} =
\seqsplit{%
100016001200056001820043680800811440128701144008008043680182000560001200001600001} 
$\\[2ex]
This $16^\text{th}$ row can also be verified from the existing Pascal's triangle. The above formula can be used for a large $n$. We now exemplify $51^{st}$ row of Pascal's triangle. Hence $n=51$ gives $\Theta=14$\\[2ex]
We have to put $14$ zeros between $1~\text{and}~1~$ in 11, that is
$(1000000000000001)^{51}$.\\[2ex]  
%%%%%%%%%%%%%%%%%%%%%%%%%%%%%%%%%%%%%%%%%%%%%%%%%%
$(1000000000000001)^{51} =
1
{\color{red}000000000000051}
{\color{blue}000000000001275}
{\color{red}000000000020825}\\
{\color{blue}000000000249900}
{\color{red}000000002349060}
{\color{blue}000000018009460}
{\color{red}000000115775100}
{\color{blue}0000006367630}\\
{\color{blue}50}
{\color{red}000003042312350}
{\color{blue}000012777711870}
{\color{red}000047626016970}
{\color{blue}000158753389900}
{\color{red}00047626016}\\
{\color{red}9700}
{\color{blue}001292706174900}
{\color{red}003188675231420}
{\color{blue}007174519270695}
{\color{red}014771069086725}
{\color{blue}027900908}\\
{\color{blue}274925}
{\color{red}048459472266975}
{\color{blue}077535155627160}
{\color{red}114456658306760}
{\color{blue}156077261327400}
{\color{red}1967930}\\
{\color{red}68630200}
{\color{blue}229591913401900}
{\color{red}247959266474052}
{\color{blue}247959266474052}
{\color{red}229591913401900}
{\color{blue}19679}\\
{\color{blue}3068630200}
{\color{red}156077261327400}
{\color{blue}114456658306760}
{\color{red}077535155627160}
{\color{blue}048459472266975}
{\color{red}027}\\
{\color{red}900908274925}
{\color{blue}014771069086725}
{\color{red}007174519270695}
{\color{blue}003188675231420}
{\color{red}001292706174900}
{\color{blue}0}\\
{\color{blue}00476260169700}
{\color{red}00158753389900}
{\color{blue}000047626016970}
{\color{red}000012777711870}
{\color{blue}000003042312350}\\
{\color{red}000000636763050}
{\color{blue}000000115775100}
{\color{red}000000018009460}
{\color{blue}000000002349060}
{\color{red}0000000002499}\\
{\color{red}00}
{\color{blue}000000000020825}
{\color{red}000000000001275}
{\color{blue}000000000000051}
{\color{red}000000000000001}
$
\newline\newline
The desired $51^{\text{st}}$ row can be obtained by partitioning each $15$ digits from the right. For readers convenience, we marked each segment with different colors and showing that the above formula generates the $51^{\text{st}}$ row of the Pascal's triangle. Now we give a proof of the fact observed above.

We have the following inequalities
\begin{flalign}
\Theta=\lfloor\log_{10} \binom{n}{\lfloor\frac{n}{2}\rfloor}\rfloor
\end{flalign}
\begin{flalign*}
\Rightarrow 10^{\Theta+1} > \binom{n}{\lfloor\frac{n}{2}\rfloor}
\end{flalign*}
\begin{flalign}
&\Rightarrow 10^{\Theta+1}-1\ge \binom{n}{\lfloor\frac{n}{2}\rfloor}\\
&\binom{n}{\lfloor\frac{n}{2}\rfloor}\ge\binom{n}{r}\\
&10^{\Theta}\le\binom{n}{\lfloor\frac{n}{2}\rfloor}
\end{flalign}

And notice that $\left(1\Theta 1\right)^{n}=\left(10^{\Theta+1}+1\right)^{n}$.

\begin{lem}
\label{pascal1}
If $n, r\in\mathbb{N}$ and $0\le r\le n$, then
\begin{eqnarray*}
\binom{n}{n-(r-1)}10^{(r-1)(\Theta+1)}+\binom{n}{n-(r-2)}10^{(r-2)(\Theta+1)}+\cdots+1<10^{r(\Theta+1)}.
\end{eqnarray*}
\end{lem}

\begin{proof}
By inequality 3, we have
\begin{eqnarray*}
&\binom{n}{n-(r-1)}10^{(r-1)(\Theta+1)}+\binom{n}{n-(r-2)}10^{(r-2)(\Theta+1)}+\cdots+1\\
& <\binom{n}{\lfloor\frac{n}{2}\rfloor}\left(10^{(r-1)(\Theta+1)}+10^{(r-2)(\Theta+1)}+\cdots+1\right)\\
&=\binom{n}{\lfloor\frac{n}{2}\rfloor}\left(\frac{10^{r(\Theta+1)}-1}{10^{(\Theta+1)}-1}\right)
<\binom{n}{\lfloor\frac{n}{2}\rfloor}\left(\frac{10^{r(\Theta+1)}}{10^{(\Theta+1)}-1}\right).
\end{eqnarray*}
Dividing the last quantity by $10^{r(\Theta+1)}$ and using the inequality 2, we have
\begin{eqnarray*}
&\binom{n}{\lfloor\frac{n}{2}\rfloor}\left(\frac{10^{r(\Theta+1)}}{10^{(\Theta+1)}-1}\right)\left(\frac{1}{10^{r(\Theta+1)}}\right)=\binom{n}{\lfloor\frac{n}{2}\rfloor}\left(\frac{1}{10^{(\theta+1)}-1}\right)\\
&\le\binom{n}{\lfloor\frac{n}{2}\rfloor}\frac{1}{\binom{n}{\lfloor\frac{n}{2}\rfloor}}=1.
\end{eqnarray*}
Therefore 
\begin{eqnarray*}
\binom{n}{n-(r-1)}10^{(r-1)(\Theta+1)}+\binom{n}{n-(r-2)}10^{(r-2)(\Theta+1)}+\cdots+1<10^{r(\Theta+1)}.
\end{eqnarray*}
\end{proof}

\begin{Prop}
\label{pascal2}
If $n,r\in\mathbb{N}$ and $0\le r\le n$, then
\begin{flalign*}
&\left(10^{\Theta+1}+1\right)^{n}\mod 10^{r(\Theta+1)}\\
&=\binom{n}{n-(r-1)}10^{(r-1)(\Theta+1)}+\binom{n}{n-(r-2)}10^{(r-2)(\Theta+1)}+\cdots+1
\end{flalign*}
\end{Prop}

\begin{proof}
Expanding $\left(10^{\Theta+1}+1\right)^{n}$ by binomial theorem, we have
%\begin{eqnarray*}
\begin{flalign*}
&\left(10^{\Theta+1}+1\right)^{n}\mod 10^{r(\Theta+1)}\\
&=[10^{n(\Theta+1)}+\binom{n}{1}10^{(n-1)(\Theta+1)}+\cdots
+\binom{n}{n-r}10^{(r)(\Theta+1)}+\\
&\binom{n}{n-(r-1)}10^{(r-1)(\Theta+1)}
+\binom{n}{n-(r-2)}10^{(r-2)(\Theta+1)}+\cdots+1]
\mod 10^{r(\Theta+1)}\\
&=\binom{n}{n-(r-1)}10^{(r-1)(\Theta+1)}+\binom{n}{n-(r-2)}10^{(r-2)(\Theta+1)}+\cdots+1,
\end{flalign*}
%\end{eqnarray*}
since by the Lemma \ref{pascal1}, we have
\begin{eqnarray*}
\binom{n}{n-(r-1)}10^{(r-1)(\Theta+1)}+\binom{n}{n-(r-2)}10^{(r-2)(\Theta+1)}+\cdots+1< 10^{r(\Theta+1)}.
\end{eqnarray*}
\end{proof}

\begin{Cor}
\label{pascal3}
The integer $\binom{n}{n-(r-1)}10^{(r-1)(\Theta+1)}+\binom{n}{n-(r-2)}10^{(r-2)(\Theta+1)}+\cdots+1$ has at best $r(\Theta+1)$ significant digits.
\end{Cor}

\begin{proof}
Since $\binom{n}{n-(r-1)}10^{(r-1)(\Theta+1)}+\binom{n}{n-(r-2)}10^{(r-2)(\Theta+1)}+\cdots+1$ is the remainder when $\left(10^{(\Theta+1)}+1\right)^{n}$ is moded out by $10^{r(\Theta+1)}$.
\end{proof}

\begin{Cor}
\label{pascal4}
The left most $\left(\Theta+1\right)$ digits of the integer 
\begin{eqnarray*}
\binom{n}{n-(r-1)}10^{(r-1)(\Theta+1)}+\binom{n}{n-(r-2)}10^{(r-2)(\Theta+1)}+\cdots+1
\end{eqnarray*}
 is 
\begin{eqnarray*} 
 \binom{n}{n-(r-1)}=\binom{n}{r-1}.
\end{eqnarray*}
\end{Cor}
\begin{proof}
By Corollary \ref{pascal3}, the integer 
\begin{eqnarray*}
\binom{n}{n-(r-1)}10^{(r-1)(\Theta+1)}+\binom{n}{n-(r-2)}10^{(r-2)(\Theta+1)}+\cdots+1
\end{eqnarray*}
has at most $r(\Theta+1)$ significant digits, and similarly 
\begin{eqnarray*}
\binom{n}{n-(r-2)}10^{(r-2)(\Theta+1)}+\binom{n}{n-(r-3)}10^{(r-3)(\Theta+1)}+\cdots+1
\end{eqnarray*}
has at most $(r-1)(\Theta+1)$ significant digits. But 
$\binom{n}{n-(r-1)}10^{(r-1)(\Theta+1)}$ has $(r-1)(\Theta+1)$  zeros to the right, the left most $(r)(\Theta+1)-(r-1)(\Theta+1)=\Theta+1$ digits gives 
\begin{eqnarray*} 
 \binom{n}{n-(r-1)}=\binom{n}{r-1}.
\end{eqnarray*}
\end{proof}

\begin{thm}
The $r$-th block of $(\Theta+1)$ digits from the right of the integer $(1\Theta1)^{n}$ is the binomial coefficient $\binom{n}{r-1}$, where $\Theta=\lfloor\log_{10} \binom{n}{\lfloor\frac{n}{2}\rfloor}\rfloor$. 
\end{thm}
\begin{proof}
The proof follows from the Corollary \ref{pascal3} and the Corollary \ref{pascal4}. Hence partitioning the digits of the integer $(1\Theta1)^{n}$ generates all the binomial coefficients or the $(n+1)^{\text{th}}$ row of the Pascal's triangle.
\end{proof}

\section{Conclusion}
Sir Isaac Newton hinted that binomial coefficients in the 
$(n+1)^{\text{th}}$ row of the Pascal's triangle may be achieved from partitioning the digits in the $n^{\text{th}}$ power of some variation of $11$ \cite{newton1736}. It has been shown earlier the weighted sum of the values in the $(n+1)^{\text{th}}$ row of the Pascal's triangle is $(11)^{n}$ \cite{Arnold}. We have shown that $\left( \Theta +1\right)$ digit partition of $(1 \Theta1)^{n}$ from the right gives the values of the $(n+1)^{\text{th}}$ row of the Pascal's triangle.\\[5ex]

%\vspace{0.15in}
%\textbf{\large Declaration of competing interest:} The authors have no competing interests.

\vspace{0.25in}
\bibliographystyle{apa}

\end{document}